\documentclass[12pt]{iopart}

\usepackage{graphicx,setspace}
\usepackage{psfrag}
\usepackage{amsfonts}
\expandafter\let\csname equation*\endcsname\relax
\expandafter\let\csname endequation*\endcsname\relax
\usepackage{amsmath}
\usepackage{amssymb, amsthm}
\usepackage{mathrsfs}
\usepackage{epstopdf}
\usepackage{float}
\usepackage{lineno,hyperref}
\usepackage{color}
\usepackage{subfigure}

\usepackage{amsmath}
\usepackage{dsfont} 
\usepackage{amssymb} 
\usepackage{graphicx,color}
\usepackage{extarrows}
\usepackage{graphicx}

\usepackage{epstopdf}

\usepackage{mathrsfs} 
\usepackage{caption}
\usepackage{subfigure}

\usepackage{microtype,booktabs}

\numberwithin{equation}{section}

\modulolinenumbers[5]
\begin{document}

\newtheorem{thm}{Theorem}[section]
\newtheorem{cor}[thm]{Corollary}
\newtheorem{lem}[thm]{Lemma}
\newtheorem{lemma}[thm]{Lemma}
\theoremstyle{remark}
\newtheorem{rem}[thm]{Remark}
\newtheorem{remark}[thm]{Remark}
\theoremstyle{definition}
\newtheorem{defn}[thm]{Definition}
\newtheorem{example}[thm]{Example}
\newcommand{\ZZ}{\mathbb{Z}}
\newcommand{\QQ}{\mathbb{Q}}
\newcommand{\NN}{\mathbb{N}}
\newcommand{\RR}{\mathbb{R}}
\newcommand{\CC}{\mathbb{C}}
\newcommand{\TT}{\mathbb{T}}
\newcommand{\EE}{\mathbb{E}}
\newcommand{\PP}{\mathbb{P}}
\newcommand{\op}{\operatorname}
\providecommand{\abs}[1]{\lvert#1\rvert}
\providecommand{\norm}[1]{\lVert#1\rVert}

\renewcommand{\theequation}{\thesection.\arabic{equation}}

\title[Lyapunov Exponents for Non-Gaussian Hamiltonian Systems]{Lyapunov Exponents for Hamiltonian Systems under Small L\'evy Perturbations}

\author{Ying Chao$^1$, Pingyuan Wei$^{2,*}$ and
Jinqiao Duan$^{3}$}

\address{$^1$ School of Mathematics and Statistics, Xi’an Jiaotong University, Xi’an 710049, China}
\address{$^2$School of Mathematics and Statistics \& Center for Mathematical Sciences, Huazhong University of Science and Technology, Wuhan 430074,  China}
\address{$^3$Department of Applied Mathematics, Illinois Institute of Technology, Chicago, IL 60616, USA}

\address{$^*$ Corresponding author: weipingyuan@hust.edu.cn}
\ead{yingchao1993@xjtu.edu.cn, weipingyuan@hust.edu.cn, duan@iit.edu.}

\begin{abstract}
This work is to investigate the (top) Lyapunov exponent for a class of Hamiltonian systems under small non-Gaussian  L\'evy noise. In a suitable moving frame, the linearisation of such a system can be regarded as a small perturbation of a nilpotent linear system. The  Lyapunov exponent is then  estimated by taking a Pinsky-Wihstutz transformation and applying the Khas'minskii formula, under appropriate assumptions on smoothness, ergodicity and integrability. Finally,  two examples are present to illustrate our results.
The result characterizes the growth or decay rates of a class of dynamical systems under the interaction between Hamiltonian structures and non-Gaussian uncertainties.

\textbf{Keywords}:  Lyapunov exponent; Hamiltonian systems; (non-Gaussian) L\'evy noise; Pinsky-Wihstutz transformation; stochastic stability along a trajectory.
\end{abstract}


\section{Introduction}
The randomly influenced Hamiltonian systems are mathematical models for complex phenomena in physical, chemical and biological science \cite{Bi,Arnold1998,FW2012}. These systems are stochastic differential equations (SDEs) whose deterministic counterparts have Hamiltonian structures, and their solutions are referred as (Hamiltonian) diffusion processes. As random fluctuation (also called, random noise) is present, the original deterministic Hamiltonian structure is usually modified or destroyed, and some interesting phenomena may occur. A meaningful issue is to investigate the stochastic stability of these systems by evaluating the Lyapunov exponents of the corresponding linearized systems. 

For stochastic Hamiltonian systems driven by (Gaussian) Brownian motion, this stability problem has been studied extensively, for example, in Arnold \cite{Arnold1998}, Arnold-Oeljeklaus-Pardoux \cite{Arnold1986},  Ariaratnam-Xie \cite{Ariaratnam1990}, Baxendale-Goukasian \cite{Baxendale2001,Baxendale2002} and Zhu \cite{Zhu2004}. There are two common ways to linearize  these systems. One way is to consider the linearizations at equilibrium states and get the corresponding (top) Lyapunov exponents. Then the signs of Lyapunov exponents determine the (local) stabilities for the original systems. Within a certain region of an equilibrium state, consider two different trajectories obtained by starting the original SDE at distinct points. If the Lyapunov exponent is negative, it is clear that they will simultaneously converge to the equilibrium state. But, if the Lyapunov exponent is positive, it is not readily apparent what will happen to the distance apart of these two trajectories as time goes to infinity. This gives rise to the other way to linearize systems (i.e., along trajectories), and then examine the stochastic stability along trajectories. In this present paper, we will adopt the latter one. For more related results on Lyapunov exponents of SDEs with Brownian motion, we refer to \cite{Baxendale1999,Mao2006}.

However, the Gaussian paradigm is known to be too limited in many physical contexts. Non-Gaussian random fluctuations, including heavy-tailed distributions and burst-like events as in earthquakes or abrupt climate change, are widely observed in scientific and engineering systems \cite{Sato,Ap,Duan}. Recently, Lyapunov exponents of stochastic differential equations with (non-Gaussian) L\'evy noise have received increasing attention, as in, for example, Mao-Rodkina \cite{Mao1995}, Applebaum \cite{Ap} and Qiao-Duan \cite{QiaoDuan2016}. It is thus desirable to investigate Lyapunov exponents for Hamiltonian system evolutions under non-Gaussian L\'evy fluctuations. But very few works are available on this topic as far as we know.

Our goal in this paper is to estimate the (top) Lyapunov exponents for one degree-of-freedom Hamiltonian systems under small L\'evy perturbations. Based on interlacing technique \cite{Ap}, the L\'evy perturbations consider here consist of a Brownian part and a bounded jump part. These systems should be interpreted as modified versions of Marcus SDEs \cite{Mar,KPP} and the solution processes are jump diffusion processes. One of our main inspirations comes from the work \cite{Baxendale2002} by Baxendale and Goukasian, where the authors studied a class of stochastic Duffing-van der Pol equations, and estimated the Lyapunov exponents for this Hamiltonian system under small Gaussian perturbation. We generalize and improve their approach to deal with the case of non-Gaussian L\'evy perturbations. This gives rise to several diﬀiculties both in analytic and probabilistic aspects. 

This paper is organized as follows. In Section 2, we recall some basic facts on L\'evy motions and Hamiltonian systems with L\'evy noise. Section 3 is dedicated to formulate precise assumptions and present our main results, Theorem \ref{theorem}. We first show that the linearized equations are randomly influenced nilpotent linear systems. Then we introduce a Pinsky-Wihstutz transformation and apply a Khas'minskii formula to estimate the Lyapunov exponent. Finally, we present two specific illustrative examples in Section 4.

\section{Preliminaries}

We first recall some basic facts on L\'evy motion \cite{Ap, Sato, Duan}, and introduce a class of stochastic Hamiltonian systems to be studied.

\subsection{L\'evy Motions} 

Let $( \Omega , \mathscr{F}, \{ \mathscr{F}_t \}_{t\geqslant0} , \PP)$ be a filtered probability space, where $\mathscr{F}_t$ is a nondecreasing family of sub-$\sigma$-fields of $\mathscr{F}$ satisfying the usual conditions. An $\mathscr{F}_t$-adapted stochastic process $L_t=L(t)$ taking values in $\mathbb{R}^d$ with $L(0)=0$ $a.s.$ (almost surely) is called a L\'{e}vy motion if it is stochastically continuous,   with independent increments and stationary increments.\par

An $d$-dimensional L\'{e}vy motion can be characterized by a drift vector $b\in \mathbb{R}^d$, an $d \times d$ non-negative-definite, symmetric matrix $Q$,  and a Borel measure $\nu$ defined on ${\mathbb{R}^d}\backslash \{ 0\}$. We call $(b,Q,\nu)$ the generating triplet of the L\'{e}vy motion $L_t$ and write this L\'{e}vy motion as $L_t\sim (b,Q,\nu)$ for convenient. Moreover, we have the following L\'{e}vy-It\^o decomposition for $L_t$:
\begin{equation}\label{decomposition}
{L_t}= bt + B_{Q}(t) + \int_{\|z\|< c} z \tilde N(t,dz) + \int_{\|z\|\ge c} z N(t,dz),
\end{equation}
where $N(dt,dz)$ is the Poisson random measure on $\mathbb{R}^{+}\times({\mathbb{R}^d}\backslash \{ 0\})$, $\tilde N(dt,dz) = N(dt,dz) - \nu (dz)dt$ is the compensated Poisson random measure, $\nu= \mathbb{E}N(1,\cdot)$ is the jump measure, $B_Q(t)$ is an independent $d$-dimensional Brownian motion with covariance matrix $Q$, and the last two terms describe the ‘small jumps’ and ‘big jumps’ of Lévy process, respectively. Here $\|\cdot\|$ denotes the Euclidean norm and $c$ is a positive constant. In the following, we denote $L_c(t)=\gamma t + B_{Q}(t)$ as the continuous part of $L_t$ and $L_d(t)=L_t-L_c(t)$ as the discontinuous part.
\par

\subsection{Hamiltonian Systems with L\'evy noise}

Let $H:\RR^2\to\RR$ be a smooth function with isolated critical points such that $H(x)\to\infty$ as $\|x\|\to\infty$, and $U_1=(\partial H/ \partial x^2,-\partial H/ \partial x^1)^T$ be the Hamiltonian vector field with respect to $H$. Given smooth vector fields $V_1,\cdots,V_d$ and a $d$-dimensional L\'evy motion $L_t=(L^{k}_t)_{k=1,\cdots, d}$ with generating triplet $(0, I,\nu_{})$. For $0<\varepsilon<1$, we consider the following one degree-of-freedom Hamiltonian system with L\'evy noise
\begin{equation}\label{Equation-1}
dx_t=U_1(x_t) dt+\varepsilon \sum_{k=1}^d V_k (x_t) \diamond dL^{k}_t,\;\;x(0)=x_0,
\end{equation}
or equivalently,
\begin{equation}\label{Equation-1-1}
x_t=x_0+\int_0^t U_1(x_s)ds+\varepsilon \sum_{k=1}^d \int_0^t V_k (x_{s-}) \diamond dL^{k}_s,
\end{equation}
where ``$\diamond$" indicates Marcus (canonical) integral \cite{Mar,KPP} defined by
\begin{align}
\varepsilon\sum_{k=1}^d\int_0^t V_k (x_{s-}) \diamond dL^{k}_s=&\varepsilon\sum_{k=1}^d\int_0^t V_k (x_{s-}) \circ dL^{k}_{c,s} +\varepsilon\sum_{k=1}^d\int_0^t V_k (x_{s-})dL^{k}_{d,s}  \notag\\
&+\sum_{0\leqslant s\leqslant t}\left[ \xi^\varepsilon(\Delta L_{s}(x_{s-})) - x_{s-}-\varepsilon\sum_{k=1}^dV_k (x_{s-})\Delta L^{k}_s
\right] \notag
\end{align}
with $\int\cdots\circ dL^k_{c,s}$ denoting the Stratonovtich integral, $\int \cdots dL^{k}_{d,s}$ denoting the It\^o integral and $\xi^\varepsilon$ being the value at $\tau= 1$ of the solution of the following ODE: (for each $x\in\RR^2$ and $z\in\RR^d$)
\begin{align}\label{Marcus-1-1}
\frac{d\xi^\varepsilon (\tau z)(x)}{d\tau}=\varepsilon\sum_{k=1}^dz_k V_k (\xi^\varepsilon (\tau z)(x)),\;\; \xi^\varepsilon(0)(x)=x.
\end{align}
Compared with It\^o stochastic differentials, Marcus one has the advantage of leading to ordinary chain rule under a transformation (change of variable). This offers some slight reduction in the lengths of some of the calculations, but has no serious effect on the theorem. We also remark that \eqref{Equation-1} is referred as a stochastic Hamiltonian system preserving symplectic structure if $V_k$, $k=1,\cdots d$ are Hamiltonian vector fields \cite{Pingyuan2019}. 

Rewritting SDE \eqref{Equation-1} into It\^o form, we have
\begin{align}\label{Equation-1-ito}
dx_t=& U_1(x_t) dt +\frac{1}{2}\varepsilon^2\sum_{k=1}^d (DV_kV_k)(x_t) dt+\varepsilon\sum_{k=1}^d V_k(x_t)dB_t^k\notag\\
&+\int_{\|z\|<c}\left[\xi^\varepsilon(z)(x_{t-})-x_{t-}-\varepsilon\sum_{k=1}^dz_k V_k(x_{t-}) \right]\nu(dz)dt\notag\\
&+\int_{\|z\|<c}\left[\xi^\varepsilon(z)(x_{t-})-x_{t-}\right]\tilde{N}(dt,dz)\notag\\
&+\int_{\|z\|\geqslant c}\left[\xi^\varepsilon(z)(x_{t-})-x_{t-}\right]{N}(dt,dz), 
\end{align}
where $c$ is a positive constant. The term in \eqref{Equation-1-ito} involving large jump is controlled by a function on ${\{\|z\|\geqslant c\}}$, and it would be absent if we take $c=\infty$. One standard way to handle this case is to use interlacing technique \cite{Ap}. In this paper, we use this technique directly without making precise statements. For more details about interlacing, we would like to refer to Applebaum \cite{Ap}. Hence, we can omit the large-jump terms and concentrate on the study of the equations driven by continuous noise interspersed with small jumps, then SDE \eqref{Equation-1-ito} can be modified as
\begin{align}\label{Equation-1-modified}
dx_t=& U_1(x_t) dt +\frac{1}{2}\varepsilon^2\sum_{k=1}^d (DV_kV_k)(x_t) dt+\varepsilon\sum_{k=1}^d V_k(x_t)dB_t^k\notag\\
&+\int_{\|z\|<c}\left[\xi^\varepsilon(z)(x_{t-})-x_{t-}-\varepsilon\sum_{k=1}^dz_k V_k(x_{t-}) \right]\nu(dz)dt\notag\\
&+\int_{\|z\|<c}\left[\xi^\varepsilon(z)(x_{t-})-x_{t-}\right]\tilde{N}(dt,dz).
\end{align}

If we assume that $U_1(x)$ and $\tilde{V}(x)\triangleq\sum_{k=1}^d (DV_kV_k)(x) $ are locally Lipschitz continuous functions and satisfy `one sided linear growth’ condition in the following sense:
\begin{enumerate}
\item[(A1)] (Locally Lipschitz condition) For any $R>0$, there exists $K_1 > 0$ such that, for all $\|x_1\|$, $\|x_2\|\leqslant R$, 
$$
\|U_1(x_1)-U_1(x_2)\|^2+\|\tilde{V}(x_1)-\tilde{V}(x_2)\|^2+\max_{1\leqslant k \leqslant d}\|V_k(x_1)-V_k(x_2)\|^2\leqslant K_1\|x_1-x_2\|^2,
$$
\item[(A2)] (One sided linear growth condition) There exists $K_2 > 0$ such that, for all $x\in\RR^2$, 
$$
\sum_{k=1}^d\|V_k(x)\|^2+2x\cdot (U_1+\tilde{V})(x)\leqslant K_2(1+\|x\|^2),
$$
\end{enumerate}
then there exists a unique solution to \eqref{Equation-1-modified} and the solution process is adapted and c\'adl\`ag, based on Lemma 6.10.3 of \cite{Ap} and Theorem 3.1 of \cite{Br}. Referring to Theorem 6.8.2 of \cite{Ap}, assumption (A2) also ensures that the solution process of \eqref{Equation-1-modified} has a Lyapunov exponent. 


\section{Lyapunov Exponent for A Hamiltonian System under Small L\'evy Perturbation}
From now on, we focus on the modified SDE \eqref{Equation-1-modified}, which is indeed a one degree-of-freedom Hamiltonian system with small L\'evy perturbation. For convenient, we rewrite SDE \eqref{Equation-1-modified} into canonical form:
\begin{equation}\label{Equation-1-modified-marcus}
dx_t=U_1(x_t) dt+\varepsilon \sum_{k=1}^d V_k (x_t) \diamond d\tilde{L}^{k}_t
\end{equation}
with $\tilde{L}^{k}_t$ a L\'evy-type noise neglecting large jump part. Linearizing the canonical SDE \eqref{Equation-1-modified-marcus} along the trajectory $x_t$, we get the following equation 
\begin{equation}\label{Equation-2}
dv_t=DU_1(x_t)v_t dt+\varepsilon\sum_{k=1}^d DV_k (x_t)v_t \diamond d\tilde{L}^{k}_t.
\end{equation}
Instead of calculate the Lyapunov exponent of such a linearization system directly, we would like to estimate it by taking advantage of the specific structure of equation \eqref{Equation-2}.

\subsection{Nilpotent Structure of the Linearization System}

We first claim that, when rewritten with respect to a suitable moving frame,  the linearization of  perturbed Hamiltonian system \eqref{Equation-1-modified-marcus}, that is, system \eqref{Equation-2}, is a perturbed nilpotent linear system. For convenience, we write 
$
U_2(x)=\nabla H(x)/\|\nabla H(x)\|^2,
$ 
and adopt the notation $U. f (x) = Df(x)(U(x)) = \langle \nabla f (x), U(x)\rangle$, which is the action of a vector field $U$ as a first order differential operator acting on a function $f$. 

Denote by $M$ the space $\RR^2$ with all critical points of $H$ removed. Note that, for $x\in M$, $(U_1.H)(x)=0$, $(U_2.H)(x)=1$ and $\langle U_1 (x), U_2(x)\rangle=0$. We rewrite smooth vector fields $V_k$, $k=1,\cdots, d$, into 
\begin{equation}\label{Vk}
{V}_k(x)=a_k^1(x)U_1(x)+a_k^2(x)U_2(x),\;\;x\in M,
\end{equation}
where $a_k^1$ and $a_k^2$ are smooth functions with respect to $x$. 
Then the original system \eqref{Equation-1-modified-marcus} becomes 
\begin{equation}\label{Equation-3}
dx_t=U_1(x_t) dt+\varepsilon\sum_{k=1}^d (a_k^1U_1+a_k^2U_2)(x_t) \diamond d\tilde{L}^{k}_t.
\end{equation}
Furthermore, we have the following lemma:

\begin{lemma}  \label{lem-2}
Let $w_t=(w_{t}^1,w_{t}^2)$ represent the linearized process $v_t$ in a moving frame given by $U_1(x)$ and $U_2(x)$. That is,
\begin{equation}  \label{eq-3}
{v}_t=w_{t}^1U_1(x_t)+w_{t}^2U_2(x_t)
\end{equation}
Define exist time $\tau_\varepsilon=\inf_{t\geqslant 0}\{ x_t \notin M \text{ or  }  x_t=\pm\infty \}$ as the first time that the process $x_t$ given by \eqref{Equation-1-modified-marcus} hits a critical point of $H$ or explodes to infinity. For all $t<\tau_\varepsilon$,  we have
\begin{equation}\label{Equation-5}
dw_t=\Lambda(x_t)w_tdt+\varepsilon\sum_{k=1}^d M_k(x_t)w_t \diamond d\tilde{L}^{k}_t,
\end{equation}
where $\Lambda=\begin{bmatrix}\begin{matrix}0 & A \\ 0&  0\end{matrix}\end{bmatrix}$, $M_k=\begin{bmatrix}\begin{matrix} B_k & C_k \\ D_k &  E_k \end{matrix}\end{bmatrix}$ $(k=1,\cdots,d)$ with 
\begin{align}
&A(x)=\big\{\big[\big((\partial_2H)^2-(\partial_1H)^2\big)(\partial_{22}H-\partial_{11}H)+4\partial_1H\partial_2H\partial_{12}H\big]/\|\nabla H\|^2\big\} (x), \notag\\
 &B_k(x)=(U_{1}.a_k^1)(x)-A(x)a_k^2(x),\;\;\;\;\;\;\; D_k(x)=(U_{1}.a_k^2)(x), \notag\\
& C_k(x)= {(U_{2}.a_k^1)(x)+A(x)a_k^1(x)},\;\;\;\; \;\;\; E_k(x)=(U_{2}.a_k^2)(x). \notag
\end{align}
\end{lemma}
\begin{proof}
The idea is to apply change of variable formula in the sense of canonical differential which could be referred to Proposition 4.3 of \cite{KPP}, and to follow the lines of \cite[Lemma 3]{Baxendale2002}. By substituting \eqref{Vk} in the equation \eqref{Equation-2}, we have
\begin{equation}
dv_t=DU_1(x_t)v_t dt+\varepsilon\sum_{k=1}^d \sum_{i=1}^2 \big(a_k^i DU_i+U_i[\nabla a_k^i]^T\big)(x_t) v_t \diamond d\tilde{L}^{k}_t. \notag
\end{equation}
And then using \eqref{eq-3}, we obtain
\begin{align}
dv_t\label{dv1}
=&\sum_{j=1}^2w_{t}^jDU_1U_j(x_t)dt +\varepsilon\sum_{k=1}^d  \sum_{i,j=1}^2  w_{t}^j\big(a_k^i DU_iU_j +U_j.a_k^i U_i\big)(x_t)\diamond d\tilde{L}^{k}_t. 
\end{align}
On the other hand, by differentiating \eqref{eq-3} with respect to $t$ and taking \eqref{Equation-3} into account, we conclude that
\begin{align}\label{dv2}
d{v}_t
=&\sum_{j=1}^2 dw_{t}^jU_j(x_t)+\sum_{j=1}^2w_t^jDU_j dx_t \notag\\
=&\sum_{j=1}^2 dw_{t}^jU_j(x_t)+\sum_{j=1}^2w_t^jDU_j \left(U_1(x_t) dt+\varepsilon\sum_{k=1}^d \sum_{i=1}^2 a_k^iU_i(x_t) \diamond d\tilde{L}^{k}_t\right) \notag\\
=&\sum_{j=1}^2 dw_{t}^jU_j(x_t)+\sum_{j=1}^2w_t^j(DU_j U_1)(x_t) dt+\varepsilon\sum_{k=1}^d  \sum_{i,j=1}^2 w_t^j(a_k^iDU_jU_i)(x_t) \diamond d\tilde{L}^{k}_t). 
\end{align}
Note that $(DU_1U_2-DU_2U_1)(x)=(AU_1)(x)$, as shown in Lemma 1 of \cite{Baxendale2002}. We equate these two semimartingale expressions for $dv_t$, that is, \eqref{dv1}-\eqref{dv2}. The result follows by comparing their coefficient terms and directly calculation. 
\end{proof}

\subsection{A L\'evy-type Pinsky-Wihstutz Transformation}

Lemma \ref{lem-2} shows that we are indeed dealing with a small L\'evy-type perturbation of a nilpotent system. 
 Motivated by the remarkable work on small random perturbation of a nilpotent system in Pinsky-Wihstutz \cite{Pinsky1988} and its applications in \cite{Baxendale2001,Baxendale2002}, we define
\begin{equation}\label{Pinsky-Wihstutz transformation}
T=\begin{bmatrix}\begin{matrix}\varepsilon^{\beta} & 0 \\ 0&  1\end{matrix}\end{bmatrix},\;\;0<\beta<1,
\end{equation}
for fixed $\varepsilon>0$. 
Since $$\min(\varepsilon^{\beta},1)\|w\|\leqslant\|Tw\|\leqslant\max(\varepsilon^{\beta},1)\|w\|,$$ it follows that $$\lim_{t\to\infty}\frac{1}{t}\log\|Tw_t\|=\lim_{t\to\infty}\frac{1}{t}\log\|w_t\|$$ and so the process $Tw_t$ has the same Lyapunov exponent as $w_t$. For simplicity of notation we continue to write $w_t$ in place of $Tw_t$. 

\begin{remark}\label{PW-Rk}
For Brownian case, that is, $L_t\sim(0,I,0)$, the value of $\beta$ in transformation \eqref{Pinsky-Wihstutz transformation} is equal to $2/3$ \cite{Pinsky1988}. Such a transformation can be used to convert a problem involving singular perturbation into one involving a non-singular perturbation; see, for example, \cite{Baxendale2001,Baxendale2002}. Moreover, after the transformation, we have a pair process $\{ (x_t,w_t):t\geqslant 0 \}$ with $x_t$ moves at the fast rate 1 and $w_t$ moving at slow rate $\varepsilon^{\beta} $. It means that, assuming enough ergodicity and integrability, a stochastic averaging argument is possible.\\


\end{remark}

Keeping in mind the equation for $x_t$ in \eqref{Equation-1-modified} or equivalently \eqref{Equation-1-modified-marcus}, the process $w_t$ $(t<\tau_\varepsilon)$ is now given by 

\begin{equation}\label{Equation-5T}
dw_t=\varepsilon^{\beta}\Lambda(x_t)w_tdt+\varepsilon\sum_{k=1}^d M_k^\varepsilon(x_t)w_t \diamond d\tilde{L}^{k}_t,
\end{equation}
where $M_k^\varepsilon=TM_kT^{-1}=\begin{bmatrix}\begin{matrix} B_k & \varepsilon^{\beta}C_k \\ \varepsilon^{-\beta}D_k &  E_k \end{matrix}\end{bmatrix}$ for $k=1,\cdots,d$. 

Note that canonical differential satisfies the change of variable formula \cite{KPP}, we write 
$$
w_t=\|w_t\|\begin{bmatrix}\begin{matrix}\cos \theta_t \\ \sin \theta_t \end{matrix}\end{bmatrix},
$$ 
and define $\rho_t=\log\|w_t\|$. Hence, equation \eqref{Equation-5T} can be rewritten as a pair of equations:
\begin{align}
d\theta_t=&-\varepsilon^\beta A(x_t)\sin^2\theta_t+\sum_{k=1}^d\sigma_k^1(x_t,\theta_t)\diamond d\tilde{L}^{k}_t,\label{theta-M}\\
d\rho_t=&\varepsilon^\beta A(x_t) \sin\theta_t\cos\theta_t+\sum_{k=1}^d\sigma_k^2(x_t,\theta_t)\diamond d\tilde{L}^{k}_t, \label{rho-M}
\end{align}
where, for $k=1,2,\cdots,d$,
\begin{align}
\sigma_k^1(x,\theta)=&\varepsilon^{1-\beta}D_k(x)\cos^2\theta-\varepsilon(B_k-E_k)(x)\sin\theta\cos\theta-\varepsilon^{1+\beta}C_k(x)\sin^2\theta \notag\\
\triangleq& \varepsilon^{1-\beta} Q_1^k(x,\theta)+\varepsilon^{} Q_2^k(x,\theta) +\varepsilon^{1+\beta} Q_3^k(x,\theta),\label{sigma1}\\
\sigma_k^2(x,\theta)=& \varepsilon^{1-\beta}D_k(x)\sin\theta\cos\theta +\varepsilon [B_k(x)\cos^2\theta+E_k(x)\sin^2\theta]+\varepsilon^{1+\beta}C_k(x)\sin\theta\cos\theta \notag\\
\triangleq&\varepsilon^{1-\beta} P_1^k(x,\theta)+\varepsilon P_2^k(x,\theta) +\varepsilon^{1+\beta} P_3^k(x,\theta). \label{sigma2}
\end{align}
Notice that, under the change of variables, the vector fields $\sigma_k^i$, $k=1,\cdots,d$, $i=1,2$, should be understood as functions of $x$, $\theta$ and $\rho$, but they indeed do not depend on $\rho$. The fact that the equation for $\log\|w_t\|$ does not involve $\|w_t\|$ was first used by Khas’minskii \cite{Khasminskii1967} in the 1960s, and is an essential part of the derivation of the Furstenberg-Khas’minskii formula for the top Lyapunov exponent of a linearized system.

We next further convert Marcus equations \eqref{theta-M}-\eqref{rho-M}  into the It\^o form \cite[Page 418]{Ap}. It is important to note that the Marcus interpretation for SDE \eqref{theta-M}-\eqref{rho-M} depends also on the SDE \eqref{Equation-3} for $x_t$. More specifically, the Marcus interpretation is only for SDE describing a diffusion process \cite{KPP}, and the pair process $\{(\theta_t,\rho_t) : t \geqslant 0\}$ is not a diffusion process because its distribution depends also on $\{x_t : t \geqslant 0\}$. A jump in the L\'evy process $L_t^\alpha$ causes jumps in both $x_t$ and $(\theta_t,\rho_t)$, we thus need to supplement ODE \eqref{Marcus-1-1} with the following equations:
\begin{align}\label{Marcus-1-2}
\frac{d\zeta_i^\varepsilon (\tau z)(x,\theta)}{d\tau}=\sum_{k=1}^dz_k  \sigma_k^i\big(\xi^\varepsilon(\tau z)(x),\zeta_1^\varepsilon(\tau z)(x,\theta)\big),\;\;i=1,2,
\end{align}
with  $\zeta^\varepsilon(0)(x,\theta)=(x,\theta)$. That is, we need to focus on the time one flow $\eta^\varepsilon(z)(x,\theta,\rho)=(\xi^\varepsilon,\zeta_1^\varepsilon,\rho+\zeta_2^\varepsilon)$ started at $(x,\theta,\rho)$ along the vector field $\sum_{k=1}^dz_k \tilde{V}_k(x,\theta,\rho)$, where
\begin{align}\label{Marcus-1-2-v}
 \tilde{V}_k(x,\theta,\rho)=\Big(\varepsilon V_k(x),\sigma_k^1(x,\theta),\sigma_k^2(x,\theta)\Big).
\end{align}
Therefore, we have the following It\^o SDE:
\begin{align}
d\theta_t=&-\varepsilon^\beta A(x_t)\sin^2\theta_tdt +\frac{1}{2}\sum_{k=1}^d \tilde{\sigma}_k^1(x_t,\theta_t) dt+\sum_{k=1}^d \sigma_k^1(x_t,\theta_t)dB_t^k\notag\\
&+\int_{\|z\|<c}\left[\zeta_1^\varepsilon(z)(x_{t-},\theta_{t-})-\theta_{t-}-\sum_{k=1}^dz_k \sigma_k^1(x_{t-},\theta_{t-}) \right]\nu(dz)dt\notag\\
&+\int_{\|z\|<c}\left[\zeta_1^\varepsilon(z)(x_{t-},\theta_{t-})-\theta_{t-}\right]\tilde{N}(dt,dz),  \label{ito-theta}\\
d\rho_t=&\varepsilon^\beta A(x_t) \sin\theta_t\cos\theta_tdt+\frac{1}{2}\sum_{k=1}^d \tilde{\sigma}_k^2(x_t,\theta_t) dt+\sum_{k=1}^d \sigma_k^2(x_t,\theta_t)dB_t^k\notag\\
&+\int_{\|z\|<c}\left[\zeta_2^\varepsilon(z)(x_{t-},\theta_{t-})-\sum_{k=1}^dz_k\sigma_k^2(x_{t-},\theta_{t-}) \right]\nu(dz)dt \notag\\
&+\int_{\|z\|<c}\left[\zeta_2^\varepsilon(z)(x_{t-},\theta_{t-})\right]\tilde{N}(dt,dz) \label{ito-rho}
\end{align}
with Wong-Zakai correction terms
 \begin{align}
\tilde{\sigma}_k^1(x,\theta)=&\varepsilon D_x \sigma_k^1 V_k+D_\theta \sigma_k^1 \sigma_k^1 \notag\\
=&\varepsilon^{2-2\beta}  D_\theta Q_1^k Q_1^k+\varepsilon^{2-\beta} (D_x Q_1^kV_k+D_\theta Q_1^kP_2^k+D_\theta Q_2^kQ_1^k) \notag\\
&+\varepsilon^{2} (D_xQ_2^kV_k+D_\theta Q_1^kQ_3^k+D_\theta Q_3^kQ_1^k) \notag\\
&+\varepsilon^{2+\beta} (D_xQ_3^kV_k +D_\theta Q_2^k Q_3^k+ D_\theta Q_3^kQ_2^k) +\varepsilon^{2+2\beta} D_\theta Q_3^k Q_3^k \notag\\
\triangleq&\varepsilon^{2-2\beta} \tilde{Q}_1^k(x,\theta)+\varepsilon^{2-\beta} \tilde{Q}_2^k(x,\theta) +\varepsilon^{2} \tilde{Q}_3^k(x,\theta)+\varepsilon^{2+\beta} \tilde{Q}_4^k(x,\theta)+\varepsilon^{2+2\beta} \tilde{Q}_5^k(x,\theta), \notag\\
  \tilde{\sigma}_k^2(x,\theta)=& \varepsilon D_x \sigma_k^2 V_k+D_\theta \sigma_k^2 \sigma_k^1 \notag\\
=&\varepsilon^{2-2\beta} D_\theta P_1^kQ_1^k+\varepsilon^{2-\beta} (D_x P_1^kV_k+D_\theta P_1^kQ_2^k+D_\theta P_2^kQ_1^k) \notag\\
&+\varepsilon^{2} (D_xP_2^kV_k +D_\theta P_1^kQ_3^k+D_\theta P_3^kQ_1^k) \notag\\
&+\varepsilon^{2+\beta} (D_xP_5^kV_k +D_\theta P_2^kQ_3^k+D_\theta P_3^kQ_2^k ) +\varepsilon^{2+2\beta} D_\theta P_3^kQ_3^k \notag\\
\triangleq&\varepsilon^{2-2\beta} \tilde{P}_1^k(x,\theta)+\varepsilon^{2-\beta} \tilde{P}_2^k(x,\theta) +\varepsilon^{2} \tilde{P}_3^k(x,\theta)+\varepsilon^{2+\beta} \tilde{P}_4^k(x,\theta)+\varepsilon^{2+2\beta} \tilde{P}_5^k(x,\theta). \notag
\end{align}

\subsection{Estimating the Lyapunov Exponent}

Based on Pinsky-Wihstutz transformation \eqref{Pinsky-Wihstutz transformation}, we now proceed to estimate the Lyapunov exponent for the linearized version of the perturbed Hamiltonian system \eqref{Equation-1-modified-marcus}. We first note that, for the process $\{(x_t, \theta_t): t \geqslant 0\}$ given by \eqref{Equation-1-modified} and \eqref{ito-theta}, its generator $\mathcal{L}_\varepsilon$ can be written as
\begin{align}\label{operator}
(\mathcal{L}_\varepsilon f)(x,\theta)=&\bigg[U_1(x)+\frac{1}{2}\varepsilon^2\sum_{k=1}^d (DV_kV_k)(x)\bigg]\cdot\nabla_x  f(x,\theta)\notag\\
&+\bigg[-\varepsilon^\beta A(x)\sin^2\theta+\frac{1}{2}\sum_{k=1}^d \tilde{\sigma}_k^1(x,\theta)\bigg]\frac{\partial}{\partial \theta}f(x,\theta)\notag\\
&+\frac{1}{2}\sum_{k=1}^d \bigg(\varepsilon^2\text{Tr}[V_k^TV_k \nabla_x\nabla_x]+2\varepsilon V_k\sigma_k^1(x,\theta) \cdot \nabla_x\frac{\partial}{\partial\theta}+({\sigma}_k^1(x,\theta))^2\frac{\partial^2 }{\partial^2 \theta}\bigg)
 f(x,\theta)\notag\\
&+\int_{\|z\|<1}\Bigg\{f\big(\xi^\varepsilon(z)(x),\zeta_1^\varepsilon(z)(x,\theta)\big)-f(x,\theta)\notag\\
&\;\;\;-\sum_{k=1}^dz_k \Big[\varepsilon V_k(x)\cdot \nabla_x f(x,\theta)+\sigma_k^1(x,\theta)\frac{\partial}{\partial \theta}f(x,\theta)\Big]\Bigg\}\nu(dz),
\end{align}
for each $f\in C_b^2(M\times S^1)$ (i.e., $f$ is a $C^2$ and bounded function). Referring to \cite{Al,Br} and keeping assumptions (A1)-(A2) in mind, we assume throughout the paper the following:

\begin{enumerate}
\item[(A3)] For each sufficiently small $\varepsilon>0$, the process $\{(x_t, \theta_t): t \geqslant 0\}$ given by \eqref{Equation-1-modified} and \eqref{ito-theta} is a positive recurrent diffusion process on $M\times S^1$ with a (unique) stationary probability $\mu^\varepsilon$. We write $\mu_M^{\epsilon}$ for the $M$ marginal of $\mu^\varepsilon$.
\end{enumerate}

From the equation \eqref{ito-rho} for $\rho_t$, we observe that the leading drift term (with respect to $\varepsilon$) of this equation, formally,  not only depends on the value of $\beta$, but also may contain contributions from the $\nu$-integral term. More specifically, it may come from the following terms: $\varepsilon^\beta A(x) \sin\theta\cos\theta$, $\frac{\varepsilon^{2-2\beta}}{2} \sum_{k=1}^d\tilde{P}_1^k(x,\theta)$
and $I_\rho^\varepsilon \triangleq \int_{\|z\|<c}\left[\zeta_2^\varepsilon(z)(x,\theta)-\sum_{k=1}^dz_k\sigma_k^2(x,\theta) \right]\nu(dz)$.

As mentioned in Remark \ref{PW-Rk}, for pure Brownian case, it has been verified that $\beta=2/3$, so the leading drift term in pure Brownian case should be 
\begin{align}\label{leading-Brownian}
&\varepsilon^{2/3} \Big(A(x) \sin\theta\cos\theta+1/2\sum_{k=1}^d\tilde{P}_1^k(x,\theta)\Big)\notag\\
&=\varepsilon^{2/3} \Big(A(x) \sin\theta\cos\theta+ \sum_{k=1}^dD_k^2(x)(\frac{1}{2}\cos^2\theta-\sin^2\theta \cos^2\theta) \Big).
\end{align}
We next try to analyze the leading term of $I_\rho^{\epsilon}$ formally. By equations \eqref{Marcus-1-1} and \eqref{Marcus-1-2}, and following calculations in the proof for Lemma 6.10.3 of \cite{Ap}, we conclude that 
\begin{align} 
I_\rho^\varepsilon(x,\theta)=&\int_{\|z\|<c}\left[\int_0^1\sum_{k=1}^d z_k\sigma_k^2\big(\xi^\varepsilon(\tau z)(x),\zeta_1^\varepsilon(\tau z)(x,\theta)\big)d\tau -\sum_{k=1}^d z_k \sigma_k^2(x,\theta)\right]\nu(dz) \notag\\
=&\int_{\|z\|<c}\int_0^1\int_0^a \sum_{k,l=1}^dz_kz_l \big(\varepsilon D_x \sigma_k^2  V_l+ D_\theta\sigma_k^2   \sigma_l^1 \big)
\big(\xi^\varepsilon(bz)(x),\zeta_1^\varepsilon(bz)(x,\theta)\big)dbda\;\nu(dz). \notag
\end{align}
Note that the expression for $\big(\varepsilon D_x \sigma_k^2  V_l+ D_\theta\sigma_k^2   \sigma_l^1 \big)(x,\theta)$ is quite similar to that of $\tilde{\sigma}_k^2(x,\theta)$. We adopt the following notation
\begin{align} \label{leading-Levy}
I_\rho^\varepsilon(x,\theta)\triangleq\int_{\|z\|<c} \left[  \varepsilon^{2-2\beta} R_0^\varepsilon+ \varepsilon^{2-\beta} R_1^\varepsilon+\varepsilon^{2} R_2^\varepsilon+\varepsilon^{2+\beta} R_3^\varepsilon+\varepsilon^{2+2\beta} R_4^\varepsilon\right] (z)(x,\theta) \nu(dz),
\end{align}
where, with $Int[f]\triangleq\int_0^1\int_0^a \sum_{k,l=1}^dz_kz_lf(\xi^\varepsilon(bz),\zeta_1^\varepsilon(bz))dbda $,
\begin{align} 
R_0^\varepsilon(z)&=Int[D_\theta P_1^kQ_1^l],\;\;\; \notag\\
R_1^\varepsilon(z)&=Int[D_x P_1^kV_l+D_\theta P_1^kQ_2^l+D_\theta P_2^kQ_1^l] \notag\\
R_2^\varepsilon(z)&=Int[D_xP_2^kV_l +D_\theta P_1^kQ_3^l+D_\theta P_3^kQ_1^l],\;\;\; \notag\\
R_3^\varepsilon(z)&=Int[D_xP_5^kV_l +D_\theta P_2^kQ_3^l+D_\theta P_3^kQ_2^l]\notag\\
R_4^\varepsilon(z)&=Int[D_\theta P_3^kQ_3^l]. \notag
\end{align}
In particular, 
\begin{align} 
R_0^\varepsilon(z)= \int_0^1\int_0^a\left(\sum_{k=1}^dz_k D_k\big(\xi^\varepsilon(bz)\big)\right)^2\left[ \frac{1}{2}\cos^2 \big(\zeta_1^\varepsilon(bz)\big)-\sin^2 \big(\zeta_1^\varepsilon(bz)\big) \cos^2 \big(\zeta_1^\varepsilon(bz)\big) \right] dbda. 
\end{align}
Since $R_i^\varepsilon$, $i=0,1,\cdots,4$, are functions depending on the value of $\varepsilon$, the integral with respect to jump measure, $I_\rho$, cannot necessarily be expanded a finite linear combination of terms of size $\varepsilon^\gamma$ for various $\gamma$.

We thus need the following assumptions on integrability and growth estimates.
\begin{enumerate}
\item[(A4)] For each sufficiently small $\varepsilon>0$, the functions $\log\|\nabla H\|$ is integrable with respect to $\mu_M^\varepsilon$; The functions 
$$
A(x) \sin\theta\cos\theta,\;\;\sum_{k=1}^d\tilde{P}_i^k(x,\theta), i=1,2,3,4,\;\;\sum_{k=1}^d{P}_j^k(x,\theta), j=1,2,3 \text{ and }  I_\rho^\varepsilon(x,\theta) 
$$
are all integrable with respect to $\mu^\varepsilon$; There is a constant $C<\infty$ such that $$\left|\int \tilde{P}_i^k(x,\theta) \mu^\varepsilon(dx,d\theta)\right|\leqslant C, \text{ for } i=2,3,4.$$
\item[(A5)] For each sufficiently small $\varepsilon>0$, the functions $R_i^\varepsilon(z)$, $i=0,1,\cdots,4,$ are all integrable with respect to $\nu$, and there is a constant
$C^{\prime}<\infty$ such that 

$$\left|\int\int_{\|z\|<c} R_i^\varepsilon(z)(x,\theta) \nu(dz) \mu^\varepsilon(dx,d\theta)\right|\leqslant C^{\prime}, \text{ for } i=0,1,\cdots,4.$$ 

\end{enumerate}

Under assumption (A4)-(A5), by \eqref{leading-Brownian} and \eqref{leading-Levy}, it makes sense for us to consider $\beta=\frac{2}{3}$ here. We note that this choice of $\beta$ is predicated on the fact that the L\'evy perturbation in \eqref{Equation-1-modified} consists of a Brownian part and a bounded jump part. To authors' knowledge, there is no standard value of $\beta$ in transformation \eqref{Pinsky-Wihstutz transformation} for the general L\'evy case so far. It is remarkable that there is some flexibility regarding choice of $\beta$ as we can always get the corresponding estimate results with slight modifications if this parameter takes different values.

We now have the following theorem on estimating the Lyapunov exponent.

\begin{thm}\label{theorem}
Let $\beta={2}/{3}$. If assumptions (A1-A5) hold, then the Lyapunov exponent for the linearized version \eqref{Equation-2} of the Hamiltonian system with L\'evy perturbation \eqref{Equation-1-modified} satisfies
\begin{align}\label{lem3}
\lambda_\varepsilon=\varepsilon^{\frac{2}{3}} \int_{M\times S^1}\Sigma_0(x,\theta)\mu^\varepsilon(dx,d\theta)+\mathcal{O}(\varepsilon^{\frac{4}{3}})
\end{align} 
as $\varepsilon\to 0$, where
\begin{align} 
\Sigma_0(x,\theta)=& A(x) \sin\theta\cos\theta+\sum_{k=1}^d D_k^2(x)(\frac{1}{2}\cos^2\theta-\sin^2\theta \cos^2\theta)+\int_{\|z\|<c} R_0^\varepsilon(z)(x,\theta) \nu(dz).
\end{align}

\end{thm}
\begin{proof}
We first claim that, for each $\varepsilon$, the processes $v_t$ and $w_t$ have the same Lyapunov exponent. 
Indeed, moving frame \eqref{eq-3} yields the inequity
$$
\min( \|\nabla H\|,\|\nabla H\|^{-1})\|w_t\|\leqslant\|v_t\|\leqslant\max( \|\nabla H\|,\|\nabla H\|^{-1})\|w_t\|.
$$
Noting that $\|\nabla H\|$ is tempered \cite[Section 4.1]{Arnold1998}, we have $\lim_{t\to\infty}\frac{1}{t}[\log\|v_t\|-\log\|w_t\|]=0$.

By equation \eqref{ito-rho}, the Lyapunov exponent of \eqref{Equation-2} is
\begin{align}
\lambda_\varepsilon=\lim_{t\to\infty}\frac{1}{t}\log\|w_t\|=\lim_{t\to\infty}\frac{1}{t}\rho_t=\lim_{t\to\infty}\frac{1}{t}\left( \log\|w_0\|+\int_0^t{P}_\varepsilon(x_s,\theta_s)ds+\mathcal{M}_t^1+\mathcal{M}_t^2\right)
\end{align}
with $${P}_\varepsilon(x,\theta)=\varepsilon^{\frac23}A(x) \sin\theta\cos\theta+\tilde{\sigma}_k^2(x,\theta) +I_\rho^\varepsilon(x,\theta),$$
$$
\mathcal{M}_t^1=\int_0^t\sum_{k=1}^d \sigma_k^2(x_s,\theta_s)dB_s^k\;\text{ and }\;\mathcal{M}_t^2=\int_0^t\int_{\|z\|<c}\left[\zeta_2^\varepsilon(z)(x_{s-},\theta_{s-})\right]\tilde{N}(ds,dz).
$$
Referring to \cite[Page 72]{Duan} and \cite[Page 109]{Ap}, both $\mathcal{M}_t^1$ and $\mathcal{M}_t^2$ are martingales for which $\mathcal{M}_t^i/t\to 0$, $i=1,2$, as $t\to\infty$. The ergodic theorem and assumption (A3) yields the Khas'minskii formula
\begin{align}
\lambda_\varepsilon= \int_{M\times S^1}P_\varepsilon(x,\theta)d\PP_\varepsilon(x,\theta).
\end{align}
The result now follows easily by using assumptions (A4) and (A5).
\end{proof}


\begin{remark}
Theorem \ref{theorem} shows that, to estimate the Lyapunov exponent for the original linearization system \eqref{Equation-2}, we could simplify the problem by analyzing the leading drift and diffusion terms of the transformed system \eqref{ito-theta}-\eqref{ito-rho} in the sense of a Pinsky-Wihstutz transformation \eqref{Pinsky-Wihstutz transformation}. This result is essentially based on the fact that the linearization of a one degree of freedom Hamiltonian system, when written with respect to a suitable moving frame, is a nilpotent linear system.
\end{remark}

\begin{remark}\label{remark3.5}
The result in Theorem \ref{theorem} depends closely on the type of perturbations which is indeed a combination of Brownian part and a  bounded jump part here. If there is no jump term, then the result in agreement with that given in \cite{Baxendale2002}. If the Brownian term is vanished and the assumption (A5) is failed, then the result may need to be modified with respect to other value of $\beta$. We also note that the assumption (A5) is almost harsh, even though it could be satisfied in some simple case (see Example \ref{ex-1}). Without this assumption, Theorem \ref{theorem} still makes sense if we replace the third term of $\Sigma_0$ by $\varepsilon^{-\frac{2}{3}}I_\rho^\varepsilon$ with $I_\rho^\varepsilon$ given in \eqref{leading-Levy}.
\end{remark}

\begin{remark}
Here we only focus on the theoretical expression of  the Lyapunov exponent, and the main novelty of our work is the model itself. We have to point out that it is still a challenging but important task to develop efficient methods for estimating the integral in \eqref{lem3}, which is left for our future research. Actually, for Brownian case, there is a remarkable method to deal with this problem by combining averaging arguments and an adjoint expansion method in non-compact spaces; see \cite{Baxendale2002} for details. However, note that the stochastic averaging theories for SDE with multiplicative L\'evy noise is still remain to solve and the corresponding generator is indeed a non-local operator, this method is unfit to the L\'evy case directly. For some interesting works on evaluating numerical values of integrals with respect to invariant measures, we refer to \cite{Ariaratnam1990,Albeverio2014,Zhangyj2020AMM} and the references therein.
\end{remark}

\section{Examples}

In this section, we present two simple illustrative examples of the theory developed in the previous section.


\begin{example}\label{ex-1} {\bf (A stochastic nilpotent linear system)}
We first recall that a symmetric $\alpha$-stable L\'evy motion  $L_t^{\alpha}$ is a special L\'evy motion, with non-Gaussianity index (or  stability index) $\alpha\in(0,2)$. It has the generating triplet $(0,0,\nu_{\alpha})$ with jump measure $\nu_{\alpha}=C_\alpha\frac{du}{|u|^{d+\alpha}}$ \cite{Duan}.

Consider the following linear SDE in $\RR^2 \backslash \{0 \} $: 
\begin{align}\label{model}
du_t=\begin{bmatrix}\begin{matrix}0 & a \\ 0&  0\end{matrix}\end{bmatrix}u_tdt+\varepsilon\begin{bmatrix}\begin{matrix}0 & 0 \\ \sigma &  0\end{matrix}\end{bmatrix}u_t\circ dB_t+\varepsilon\begin{bmatrix}\begin{matrix}0 & 0 \\ \sigma &  0\end{matrix}\end{bmatrix}u_t\diamond dL^{\alpha}_t,
\end{align}
which can be regarded as a perturbed Hamiltonian system with $H(u_1,u_2)=\frac{1}{2} a u_2^2$. Here $a>0$, $\sigma>0$ and $L^{\alpha}_t$ $(1<\alpha<2)$ is a one-dimensional standard $\alpha$-stable L\'evy motion. Note that, for this system, the time one flow is $\xi(z)(u)=(u_1,z\sigma u_1+u_2)$ and the jump term 
$\sum_{0\leqslant s\leqslant t}\left[ \xi(\Delta L^{\alpha}_s)(u_{s-}) - u_{s-}-(0,u_{s-}^1\Delta L^{\alpha}_s)\right]=0$, so that the coefficients of both Marcus and It\^o versions are identical.

By introducing the following change of variables : $(\theta \in (0,2\pi )\backslash \{ \frac{\pi}{2},\frac{3\pi}{2}\})$
$$
\cos\theta=u_1/\|u_t\|,\;\;\;\;\; \sin\theta=u_2/\|u\|, \;\;\;\;\; \rho=\log\|u\|,
$$ 
we have 
\begin{align}\label{model-theta-rho}
d\begin{bmatrix}\begin{matrix}\theta_t  \\ \rho_t\end{matrix}\end{bmatrix}=a\begin{bmatrix}\begin{matrix}-\sin^2\theta_t \\ \sin\theta_t\cos\theta_t\end{matrix}\end{bmatrix}dt
+\varepsilon\sigma\begin{bmatrix}\begin{matrix}\cos^2\theta_t \\ \sin\theta_t\cos\theta_t \end{matrix}\end{bmatrix}\circ dB_t
+\varepsilon\sigma\begin{bmatrix}\begin{matrix}\cos^2\theta_t \\ \sin\theta_t\cos\theta_t \end{matrix}\end{bmatrix}\diamond dL^{\alpha}_t,
\end{align}
or equivalently,
\begin{align}
d\theta_t=&-a\sin^2\theta_tdt-\varepsilon^2\sigma^2\sin\theta_t \cos^3\theta_tdt + \varepsilon\sigma\cos^2\theta_t dB_t\notag\\
&+\int_{|z|<c}[\zeta_1(z)(\theta_{t-})-\theta_{t-}]\nu_\alpha(dz)dt \notag\\
&+\int_{|z|<c}[\zeta_1(z)(\theta_{t-})-\theta_{t-}]\tilde{N}(dt,dz)+\int_{|z|\geqslant c}[\zeta_1(z)(\theta_{t-})-\theta_{t-}]N(dt,dz),\label{Ex-theta}\\
d\rho_t=&a\sin\theta_t\cos\theta_tdt +\varepsilon^2\sigma^2\Big(\frac{1}{2}\cos^2\theta_t-\sin^2\theta_t\cos^2\theta_t\Big)dt + \varepsilon\sigma\sin\theta_t\cos\theta_t dB_t\notag\\
&+\int_{|z|<c}[\zeta_2(z)(\theta_{t-})]\nu_\alpha(dz)dt \notag\\
&+\int_{|z|<c}[\zeta_2(z)(\theta_{t-})]\tilde{N}(dt,dz)+\int_{|z|\geqslant c}[\zeta_2(z)(\theta_{t-})]N(dt,dz),  \label{Ex-rho}
\end{align}
with 
\begin{align}
\zeta_1(z)(\theta)&=\arctan(\tan \theta +\varepsilon\sigma z),\\
\zeta_2(z)(\theta)&=\frac{1}{2}\log\left(\frac{1+(\tan\theta+\varepsilon\sigma z)^2}{1+\tan^2\theta}\right).
\end{align}

By interlacing technique, we consider the modified version of equations \eqref{Ex-theta}-\eqref{Ex-rho}. Provided the modified process $\rho(t)$ is ergodic, the Lyapunov exponent can be obtained as
\begin{align}\label{Ex-lamdda}
\lambda=\lim_{t\to\infty}\frac{1}{t}\log\|u_t\|=\lim_{t\to\infty}\frac{1}{t}\rho_t=\mathbb{E}[Q(\theta)]=\int_0^{2\pi}Q(\theta)\mu(\theta)d\theta,
\end{align}
where
\begin{align}\label{Ex-Q}
Q(\theta)=&a\sin\theta\cos\theta+\varepsilon^2\sigma^2\Big(\frac{1}{2}\cos^2\theta_t-\sin^2\theta_t\cos^2\theta_t\Big) +\int_{|z|<c}[\zeta_2(z)(\theta)]\nu_\alpha(dz),
\end{align}
and $\mu(\theta)$ is the density of the invariant measure of the modified process $\theta_t$ with respect to the uniform measure on the unit circle. Remark that the operator for modified process $\theta_t$ satisfies
\begin{align}
\mathcal{A}_\theta f(\theta)=&-(a\sin^2\theta\frac{\partial}{\partial\theta})f(\theta)+\varepsilon^2\sigma^2\Big(-\sin\theta \cos^3\theta\frac{\partial}{\partial\theta}+\frac{1}{2} \cos^4\theta \frac{\partial^2}{\partial\theta^2}\Big)f(\theta) \notag\\
&+\int_{|z|<c}\Big[ f(\zeta_1(z)(\theta))-f(\theta)\Big]\nu_\alpha(dz).
\end{align}
Referring to Sun et. al. \cite{SunDuan2012}, $\mu(\theta)$ is indeed the solution of the following stationary Fokker-Planck equation 
\begin{align}
\mathcal{A}_\theta^\ast \mu(\theta)=0,
\end{align}
that is,
\begin{align}\label{equation-Ex-FP}
&\cos^2\theta \frac{\partial}{\partial\theta}(\sin^2\theta \mu(\theta))+\frac{1}{2}\varepsilon^2\sigma^2\cos^2\theta\frac{\partial}{\partial\theta}\left(\cos^2\theta\frac{\partial}{\partial\theta}(\cos^2\theta \mu(\theta))\right) \notag\\
&~~~~~~~~~~~~~~+\int_{\mathbb{R}\backslash \{ 0 \}} \left[\cos^2(\zeta_1(z)(\theta))\mu(\zeta_1(z)(\theta))-\cos^2\theta \mu(\theta) \right]\nu_\alpha(dz)=0.
\end{align}
Notice that, by Taylor's Theorem, we have
\begin{align}
\int_{|z|<c}[\zeta_2(z)(\theta)]\nu_\alpha(dz)=&\int_{|z|<c}[g(\varepsilon\sigma z)(\theta)-g(0)(\theta)]\frac{C_\alpha}{|z|^{1+\alpha}}dz \notag\\
=&\int_{|z|<c}\left[\frac{1}{2}g^{\prime\prime}(0)(\theta)\varepsilon^2\sigma^2z^2+ \frac{1}{4!}\int_{0}^1g^{(4)}(\tau \varepsilon\sigma z)(\theta)\varepsilon^4\sigma^4z^4 d\tau \right]\frac{C_\alpha}{|z|^{1+\alpha}}dz, \notag
\end{align}
where $g(\cdot)(\theta)=\frac{1}{2}\log(1+(\tan\theta+\cdot)^2)$, $g^{\prime}(0)=\frac{\tan\theta}{1+\tan^2\theta}=\sin\theta\cos\theta$, $\frac{1}{2}g^{\prime\prime}(0)=\frac{1}{2}\frac{1-\tan^2\theta}{(1+\tan^2\theta)^2}=\frac{1}{2}\cos^2\theta-\sin^2\theta\cos^2\theta$ and $g^{(4)}(\sigma z)=\frac{-2(\tan\theta+\sigma z)^4+(\tan\theta+ \sigma z)^2+6(\tan\theta+\sigma z)-1}{[1+(\tan\theta+\sigma z)^2]^4}$. It is not hard to show that, for all $\theta\in(-\frac{\pi}{2},\frac{\pi}{2})\bigcup (\frac{\pi}{2},\frac{3\pi}{2})$ and $\varepsilon\in(0,1)$, 
there is a positive constant $C_1$ such that
\begin{align}
 \frac{1}{4!}\int_{|z|<c}\left[g^{(4)}(\delta \varepsilon\sigma z)(\theta)\sigma^4z^4 \right]\frac{C_\alpha}{|z|^{1+\alpha}}dz<C_1. \notag
\end{align}
Analogously, we have
\begin{align}
&\int_{|z|<c}\Big[ f(\zeta_1(z)(\theta))-f(\theta)\Big]\nu_\alpha(dz)=\int_{|z|<c}\Big[ f(\zeta_1(z)(\theta))-f(\zeta_1(0)(\theta))\Big]\frac{C_\alpha}{|z|^{1+\alpha}}dz \notag\\
&~~~~~~~~~~~~~~~~~~~~=\int_{|z|<c}\left[\frac{1}{2}F^{\prime\prime}(0)(\theta)\varepsilon^2\sigma^2z^2+ \frac{1}{4!}\int_{0}^1F^{(4)}(\tau \varepsilon\sigma z)(\theta)\varepsilon^4\sigma^4z^4 d\tau\right]\frac{C_\alpha}{|z|^{1+\alpha}}dz,\notag
\end{align}
where $F(\cdot)(\theta)=f\circ \arctan(\tan\theta+\cdot)$ and $\frac{1}{2}F^{\prime\prime}(0)=-\frac{\tan\theta}{(1+\tan^2\theta)^2}f^{\prime}(\theta)+\frac{1}{2}\frac{1}{(1+\tan^2\theta)^2}f^{\prime\prime}(\theta)=\big(-\cos^3\theta\sin\theta \frac{\partial}{\partial\theta}+\frac{1}{2} \cos^4\theta \frac{\partial^2}{\partial\theta^2}\big)f(\theta)$.

By taking Pinsky-Wihstutz transformation \eqref{Pinsky-Wihstutz transformation}, $\tilde{u}_t=Tu_t$, with $\beta=\frac{2}{3}$, and introducing the corresponding change of variables with respect to $(\tilde{\theta}_t,\tilde{\rho}_t)$, we conclude that 
\begin{align}\label{Ex-lamdda-2}
\lambda=\lim_{t\to\infty}\frac{1}{t}\log\|u_t\|=\lim_{t\to\infty}\frac{1}{t}\log\|Tu_t\|=\lim_{t\to\infty}\frac{1}{t}\tilde{\rho}_t.
\end{align}
According to Theorem \ref{theorem}, we further have
\begin{align}\label{Ex-lamdda-3}
\lambda=\varepsilon^{\frac{2}{3}} \int_{ S^1}\tilde{Q}(\tilde{\theta})\tilde{\mu}(\tilde{\theta})d\tilde{\theta}+\mathcal{O}(\varepsilon^{\frac{4}{3}}),
\end{align}
where
\begin{align}\label{Ex-Q-3}
\tilde{Q}(\theta)=&a\sin\theta\cos\theta+\sigma^2\Big(\frac{1}{2}\cos^2\theta-\sin^2\theta\cos^2\theta\Big) \Big(1+\int_{|z|<c}z^2\nu_\alpha(dz)\Big)\notag\\
=&a\sin\theta\cos\theta+\sigma^2\Big(1+\frac{2C_\alpha}{2-\alpha} c^{2-\alpha}\Big)\Big(\frac{1}{2}\cos^2\theta-\sin^2\theta\cos^2\theta\Big),
\end{align}
and $\tilde{\mu}(\theta)$ is the corresponding density of the invariant measure of the process $\tilde{\theta}_t$. Moreover, define
\begin{align}
\tilde{\mathcal{A}}^1_\theta=-(a\sin^2\theta\frac{\partial}{\partial\theta})+\sigma^2\Big(1+\frac{2C_\alpha}{2-\alpha} c^{2-\alpha}\Big)\Big(-\sin\theta \cos^3\theta\frac{\partial}{\partial\theta}+\frac{1}{2} \cos^4\theta \frac{\partial^2}{\partial\theta^2}\Big) .
\end{align}
Then the generator for the process $\tilde{\theta}_t$ is 
\begin{align}
\tilde{\mathcal{A}}_\theta
=&\varepsilon^{\frac{2}{3}}\tilde{\mathcal{A}}^1_\theta+\varepsilon^{\frac43} \cdot\frac{1}{4!}\sigma^4\int_{|z|<c}\int_{0}^1F^{(4)}(\tau \varepsilon\sigma z)(\theta)z^4 d\tau\frac{C_\alpha}{|z|^{1+\alpha}}dz .
\end{align}
\end{example}

\begin{example}{\bf (A stochastic Duffing equation)} This is a L\'evy version of (last) example of \cite{Baxendale2002}. We will show that the viewpoint of introducing a suitable moving frame to converting the original linearization system to a nilpotent one, and our estimation method method based on transformation \eqref{Pinsky-Wihstutz transformation} is applicable in the present example. 

Consider a stochastic L\'evy version of the Duffing equation in $\RR^2 \backslash \{0 \} $:
\begin{align}\label{Example}
\ddot{x}=-x-x^3+\varepsilon\sigma x \diamond dL_t
\end{align} 
with $\sigma>0$ and $L^{\alpha}_t \sim (0,I,\nu)$ a one-dimensional L\'evy motion. Rewriting \eqref{Example} in phase space $(x,y)=(x,\dot{x})\in\RR^2$, we have
$$
U_1(x,y)=\begin{bmatrix}\begin{matrix} y \\ -x-x^3 \end{matrix}\end{bmatrix},\;\;\; U_2(x,y)=\frac{1}{(x+x^3)^2+y^2}\begin{bmatrix}\begin{matrix} x+x^3 \\ y \end{matrix}\end{bmatrix}.
$$
This system \eqref{Example} is indeed a perturbation of the Hamiltonian system with $H(x,y)=\frac{1}{2}x^2+\frac{1}{4}x^4+\frac{1}{2}y^2$ by the following vector fields
$$
V_1(x,y)=\sigma\begin{bmatrix}\begin{matrix}0  \\ x\end{matrix}\end{bmatrix}.
$$

Linearizing system \eqref{Example} along the trajectory $(x_t,y_t)$ and introducing a moving frame in the form of \eqref{Vk} with $a_1^1(x,y)=-\frac{\varepsilon\sigma x(x+x^3)}{(x+x^3)^2+y^2}$ and $a_1^2(x,y)=\sigma xy$, we obtain 
\begin{align}
dw_t=\begin{bmatrix}\begin{matrix}0 & A \\ 0&  0\end{matrix}\end{bmatrix}w_tdt+\varepsilon\begin{bmatrix}\begin{matrix} B_1 & C_1 \\ D_1 & E_1 \end{matrix}\end{bmatrix}w_t\diamond dL_t,
\end{align}
where $A=\frac{3x^2((x+x^3)^2-y^2)}{[(x+x^3)^2+y^2]^2}$, $B_1=-E_1=\frac{-\sigma xy(x^2+2)}{(x+x^3)^2+y^2}$, $C_1=\frac{-\sigma x^3(x+x^3)}{[(x+x^3)^2+y^2]^2}$, $D_1=-\sigma x(x+x^3)+\sigma y^2$.

Let $\beta=\frac{2}{3}$, we further introduce Pinsky-Wihstutz transform \eqref{Pinsky-Wihstutz transformation} and rewrite the transformed system with respect to $w=\|w\|[\cos \theta, \sin \theta ]^T$ and $\rho=\log\|w\|$:
\begin{align}
d\theta_t=&-\varepsilon^{\frac{2}{3}} A(x_t,y_t)\sin^2\theta_t+\sigma_1^1(x_t,y_t,\theta_t)\diamond dL^{k}_t,\label{theta-M-e}\\
d\rho_t=&\varepsilon^{\frac{2}{3}} A(x_t,y_t) \sin\theta_t\cos\theta_t+\sigma_1^2(x_t,y_t,\theta_t)\diamond dL^{k}_t, \label{rho-M-e}
\end{align}
where 
\begin{align}
\sigma_1^1(u,\theta)=&\varepsilon^{\frac{1}{3}}D_1(x,y)\cos^2\theta- \varepsilon B_1(x,y)\sin 2\theta-\varepsilon^{\frac{5}{3}}C_1(x,y)\sin^2\theta, \notag\\
\sigma_1^2(u,\theta)=& \varepsilon^{\frac{1}{3}}D_1(x,y)\sin\theta\cos\theta +\varepsilon B_1(x,y)\cos 2\theta+\varepsilon^{\frac{5}{3}}C_1(x,y)\sin\theta\cos\theta. \notag
\end{align}

We next focus on the modified linearization system of  \eqref{Example} in the sense of interlacing. By Theorem  \ref{theorem} and Remark \ref{remark3.5}, the Lyapunov exponent of the modified linearization system is
\begin{align}
\lambda=\varepsilon^{\frac{2}{3}} \int_{\RR^2 \backslash \{ 0\} \times S^1}\Sigma_0(x,y,\theta)\mu^\varepsilon(dx,dy,d\theta)+\mathcal{O}(\varepsilon^{\frac{4}{3}}),
\end{align} 
where
\begin{align}\label{EX-2-sigma}
\Sigma_0(x,y,\theta)=&A(x,y) \sin\theta\cos\theta+ D_1^2(x,y)(\frac{1}{2}\cos^2\theta-\sin^2\theta \cos^2\theta) \notag\\
&+\varepsilon^{-\frac{2}{3}}\int_{\|z\|<c}\left[\zeta_2(z)(x,\theta)-z\sigma_1^2(x,\theta) \right]\nu(dz)
\end{align}
with $\zeta_2$ the time one flow of the vector field $\sigma_1^2$. According to the analysis of \eqref{leading-Levy}, the leading order of the third term in \eqref{EX-2-sigma} should not less than $\varepsilon^{0}=1$. But unfortunately, it is not easy to obtain the  precise expression of $\zeta_2$, even though we know that $\xi(\tau z)(u)=(x,\varepsilon\sigma \tau zx+y)^T$. Compared with the Brownian case in \cite{Baxendale2002}, an extra term (the third one in \eqref{EX-2-sigma}) appears in the integrand of integral expression for Lyapunov exponent. If the modified system of \eqref{Example} is only driven by Brownian noise, that is, $\nu=0$, then the result here would be consistent with that in \cite{Baxendale2002}.

\end{example}

\section*{Data Availability}
The data that support the findings of this study are available within the article.


\section*{Acknowledgements}
The authors would like to thank Dr Qiao Huang, Dr Li Lin, Dr Zibo Wang, and Dr Yanjie Zhang for helpful discussions. This work was partly supported by NSFC grants 11771449 and 11531006.

\section*{References}
\bibliographystyle{plain}
\bibliography{ref-lyapunov}

\begin{thebibliography}{10}

\bibitem{Br}
S.~Albeverio, Z.~Brze\'{z}niak, and J-L. Wu.
\newblock Existence of global solutions and invariant measures for stochastic
  differential equations driven by {P}oisson type noise with non-{L}ipschitz
  coefficients.
\newblock {\em Journal of Mathematical Analysis and Applications},
  371(1):309--322, 2010.

\bibitem{Albeverio2014}
S.~Albeverio, L.~D. Persio, E.~Mastrogiacomo, and B.~Smii.
\newblock A class of {L}\'evy driven {SDE}s and their explicit invariant
  measures.
\newblock {\em Potential Analysis}, 45(2):229--259, 2014.

\bibitem{Al}
S.~Albeverio, B.~R{\"u}diger, and J-L. Wu.
\newblock Invariant {M}easures and symmetry property of {L}{\'e}vy type
  operators.
\newblock {\em Potential Analysis}, 13(2):147--168, 2000.

\bibitem{Ap}
D.~Applebaum.
\newblock {\em L{\'e}vy {P}rocesses and {S}tochastic {C}alculus}.
\newblock Cambridge University Press, 2009.

\bibitem{Ariaratnam1990}
S.~T. Ariaratnam and W.~C. Xie.
\newblock Lyapunov exponent and rotation number of a two-dimensional nilpotent
  stochastic system.
\newblock {\em Dynamics and Stability of Systems}, 5(1):1--9, 1990.

\bibitem{Arnold1998}
L.~Arnold.
\newblock {\em Random Dynamical Systems}.
\newblock Springer, 1998.

\bibitem{Arnold1986}
L.~Arnold, E.~Oeljeklaus, and E.~Pardoux.
\newblock Almost sure and moment stability for linear {I}t\^o equations.
\newblock In {\em Lyapunov Exponents, Lecture Notes Mathematics}, pages
  129--159. Springer, 1986.

\bibitem{Baxendale1999}
P.~H. Baxendale.
\newblock Stability along trajectories at a stochastic bifurcation point.
\newblock In {\em Stochastic Dynamic}, pages 1--25. Springer, Berlin, 1999.

\bibitem{Baxendale2001}
P.~H. Baxendale and L.~Goukasian.
\newblock Lyapunov exponents of nilpotent {I}t\^o systems with random
  coefficients.
\newblock {\em Stochastic Processes and their Applications}, 95:219--233, 2001.

\bibitem{Baxendale2002}
P.~H. Baxendale and L.~Goukasian.
\newblock Lyapunov exponents for small random perturbations of {H}amiltonian
  systems.
\newblock {\em Annals of Probability}, 30(1):101--134, 2002.

\bibitem{Bi}
J-M. Bismut.
\newblock {\em M{\'e}canique Al{\'e}atoire, Lecture Notes in Mathematics},
  volume 866.
\newblock Springer-Verlag, 1981.

\bibitem{Duan}
J.~Duan.
\newblock {\em An Introduction to Stochastic Dynamics}.
\newblock Cambridge University Press, 2015.

\bibitem{FW2012}
M.~I. Freidlin and A.~D. Wentzell.
\newblock {\em Random Perturbations of Dynamical Systems, Third Edition}.
\newblock Springer-Verlag, New York, 2012.

\bibitem{Khasminskii1967}
R.~Z. Khasminskii.
\newblock Necessary and sufficient conditions for the asymptotic stability of
  linear stochastic systems.
\newblock {\em Theory of Probability \& Its Applications}, 12:144--147, 1967.

\bibitem{KPP}
T.~G. Kurtz, E.~Pardoux, and P.~Protter.
\newblock Stratonovich stochastic differential equations driven by general
  semimartingales.
\newblock In {\em Annales de l'IHP Probabilit{\'e}s et statistiques},
  volume~31, pages 351--377, 1995.

\bibitem{Mao2006}
X.~Mao.
\newblock {\em Stochastic Differential Equations and Applications, Second
  Edition}.
\newblock Horwood Publishing, 2006.

\bibitem{Mao1995}
X.~Mao and A.~E. Rodkina.
\newblock Exponential stability of stochastic differential equations driven by
  discontinuous semimartingales.
\newblock {\em Stochastics \& Stochastic Reports}, 55:207--224, 1995.

\bibitem{Mar}
S.~I. Marcus.
\newblock Modeling and approximation of stochastic differential equations
  driven by semimartingales.
\newblock {\em Stochastics: An International Journal of Probability and
  Stochastic Processes}, 4(3):223--245, 1981.

\bibitem{Pinsky1988}
M.~A. Pinsky and V.~Wihstutz.
\newblock Lyapunov exponents of nilpotent ito systems.
\newblock {\em Stochastics: formerly Stochastics and Stochastics Reports},
  25(1):43--57, 1988.

\bibitem{QiaoDuan2016}
H.~Qiao and J.~Duan.
\newblock Lyapunov exponents of stochastic differential equations driven by
  {L}\'evy processes.
\newblock {\em Dynamical Systems}, 31:136--150, 2016.

\bibitem{Sato}
K-I Sato.
\newblock {\em L{\'e}vy Processes and Infinitely Divisible Distributions}.
\newblock Cambridge University Press, 1999.

\bibitem{SunDuan2012}
X.~Sun, J.~Duan, X.~Li, H.~Liu, X.~Wang, and Y.~Zheng.
\newblock Derivation of {F}okker-{P}lanck equations for stochastic dynamical
  systems under excitation of multiplicative non-{G}aussian white noise.
\newblock {\em Journal of Mathematical Analysis \& Applications},
  446(1):786--800, 2017.

\bibitem{Pingyuan2019}
P.~Wei, Y.~Chao, and J.~Duan.
\newblock Hamiltonian systems with {L}\'evy noise: {S}ymplecticity,
  {H}amilton's principle and averaging principle.
\newblock {\em Physica D: Nonlinear Phenomena}, 398:69--83, 2019.

\bibitem{Zhangyj2020AMM}
Y.~Zhang, X.~Wang, Q.~Huang, J.~Duan, and T.~Li.
\newblock Numerical analysis and applications of {F}okker-{P}lanck equations
  for stochastic dynamical systems with multiplicative $\alpha$-stable noises.
\newblock {\em Applied Mathematical Modelling}, 87:711--730, 2020.

\bibitem{Zhu2004}
W.~Zhu.
\newblock Lyapunov exponent and stochastic stability of quasi-non-integrable
  {H}amiltonian systems.
\newblock {\em International Journal of Non-Linear Mechanics}, 39(4):569--579,
  2004.

\end{thebibliography}

\end{document}